\documentclass[12pt, reqno]{amsart}
\usepackage{amsmath}
\usepackage{amssymb}

\numberwithin{equation}{section}

\begin{document}

{\theoremstyle{plain}
    \newtheorem{theorem}{\bf Theorem}[section]
    \newtheorem{proposition}[theorem]{\bf Proposition}
    \newtheorem{claim}[theorem]{\bf Claim}
    \newtheorem{lemma}[theorem]{\bf Lemma}
    \newtheorem{corollary}[theorem]{\bf Corollary}
}
{\theoremstyle{remark}
    \newtheorem{remark}[theorem]{\bf Remark}
    \newtheorem{example}[theorem]{\bf Example}
}
{\theoremstyle{definition}
    \newtheorem{defn}[theorem]{\bf Definition}
    \newtheorem{question}[theorem]{\bf Question}
    \newtheorem{conjecture}[theorem]{\bf Conjecture}
}

%%%%%%%%%%%%%%%%%%%%%%%%%%%%%%%%%%%%%%%%%%%

\newcommand{\kset}[1]{[#1]}

\def\to{\longrightarrow}
\def\height{\operatorname{ht}}
\def\reg{\operatorname{reg}}
\def\depth{\operatorname{depth}}
\def\Hom{\operatorname{Hom}}
\def\proj{\operatorname{Proj}}
\def\grade{\operatorname{grade}}
\def\spec{\operatorname{Spec}}
\def\Tor{\operatorname{Tor}}
\def\Ext{\operatorname{Ext}}
\def\supp{\operatorname{Supp}}
\def\Ann{\operatorname{Ann}}
\def\im{\operatorname{image}}

\def\m{{\frak m}}
\def\p{{\wp}}
\def\q{{\frak q}}
\def\a{{\frak a}}
\def\d{{\frak d}}
\def\aa{{\bf a}}
\def\O{{\mathcal O}}
\def\I{{\mathcal I}}
\def\M{{\frak M}}
\def\A{{\mathcal A}}
\def\L{{\mathcal L}}
\def\NN{{\mathbb N}}
\def\N{{\mathcal N}}
\def\PP{{\mathbb P}}
\def\ZZ{{\mathbb Z}}
\def\GG{{\mathbb G}}
\def\R{{\mathcal R}}
\def\T{{\mathcal T}}
\def\F{{\mathcal F}}
\def\G{{\mathcal G}}
\def\H{{\mathcal H}}
\def\U{{\frak U}}
\def\a{{\frak a}}
\def\bx{{\widetilde{X}}}
\def\ix{{\bar{X}}}

\def\sign{\operatorname{sgn}}
\def\1{{\bf 1}}
\def\0{{\bf 0}}

%%%%%%%%%%%%%%%%%%%%%%%%%%%%%%%%%%%%%%%%%%%

\title[Asymptotic linearity of regularity and $a^*$-invariant]{Asymptotic linearity of regularity and $a^*$-invariant of powers of ideals}

\author{Huy T\`ai H\`a}  
\address{Tulane University \\ Department of Mathematics \\  
6823 St. Charles Ave. \\ New Orleans, LA 70118, USA}  
\email{tai@math.tulane.edu}  
\urladdr{http://www.math.tulane.edu/$\sim$tai/}


\begin{abstract} Let $X = \proj R$ be a projective scheme over a field $k$, and let $I \subseteq R$ be an ideal generated by forms of the same degree $d$. Let $\pi: \bx \rightarrow X$ be the blowing up of $X$ along the subscheme defined by $I$, and let $\phi: \bx \rightarrow \ix$ be the projection given by the divisor $dE_0 - E$, where $E$ is the exceptional divisor of $\pi$ and $E_0$ is the pullback of a general hyperplane in $X$. We investigate how the asymptotic linearity of the regularity and the $a^*$-invariant of $I^q$ (for $q \gg 0$) is related to invariants of fibers of $\phi$.
\end{abstract}

\maketitle

%%%%%%%%%%%%%%%%%%%%%%%%%%%%%%%%%%%%%%%%%%%%%%%%%%%%%

\section{Introduction}

Let $k$ be a field and let $X = \proj R \subseteq \PP^n$ be a projective scheme over $k$. Let $I \subseteq R$ be a homogeneous ideal. It is well known (cf. \cite{BEL, Ch, Cu, CEL, BEL, GGP, CHT, Ko, TW}) that $\reg(I^q) = aq + b$, a linear function in $q$, for $q \gg 0$. While the linear constant $a$ is quite well understood from reduction theory (see \cite{TW}), the free constant $b$ remains mysterious (see \cite{EH, R} for partial results). Recently, Eisenbud and Harris \cite{EH} showed that when $I$ is generated by general forms of the same degree, whose zeros set is empty in $X$, $b$ can be related to a set of {\it local} data, namely, the regularity of fibers of the projection map defined by the generators of $I$. The aim of this paper is to exhibit a similar phenomenon in a more general situation, when $I$ is generated by arbitrary forms of the same degree. In this case, the generators of $I$ do not necessarily give a morphism. The projection map that we will examine is the map from the blowup of $X$ along the subscheme defined by $I$, considered as a bi-projective scheme, to its second coordinate. 

Let $I = (F_0, \dots, F_m)$, where $F_0, \dots, F_m$ are homogeneous elements of degree $d$ in $R$. Let $\pi: \bx \rightarrow X$ be the blowing up of $X$ along the subscheme defined by $I$. Let $\R = R[It]$ be the Rees algebra of $I$. By letting $\deg F_it = (d,1)$, the Rees algebra $\R$ is naturally bi-graded with $\R = \bigoplus_{p,q \in \ZZ} \R_{(p,q)},$ where $\R_{(p,q)} = (I^q)_{p+qd}t^q$. Under this bi-gradation of $\R$, we can identify $\bx$ with $\proj \R \subseteq \PP^n \times \PP^m$ (cf. \cite{CH, HT}).
Also, the projection $\phi: \proj \R \rightarrow \PP^m$ is in fact the morphism given by the divisor $D = dE_0 - E$, where $E$ is the exceptional divisor of $\pi$ and $E_0$ is the pullback of a general hyperplane in $X$. For a close point $\wp \in \ix = \im(\phi)$, let $\bx_\wp = \bx \times_{\ix} \spec \O_{\ix, \wp}$ be the fiber of $\phi$ over the affine neighborhood $\spec \O_{\ix, \wp}$ of $\wp$. Then $\bx_\wp = \proj \R_{(\wp)}$, where $\R_{(\wp)}$ is the homogeneous localization of $\R$ at $\wp$. We define the {\it regularity} of $\bx_\wp$, denoted by $\reg(\bx_\wp)$, to be that of $\R_{(\wp)}$. Inspired by the work of Eisenbud and Harris \cite{EH}, we propose the following conjecture.

\begin{conjecture} \label{conj.reg}
Let $X = \proj R \subseteq \PP^n$ be a projective scheme, and let $I \subseteq R$ be a homogeneous ideal generated by forms of degree $d$. Let $\reg(\phi) = \max\{ \reg(\bx_\wp) ~|~ \wp \in \ix \}$. Then for $q \gg 0$,
$$\reg(I^q) = qd + \reg(\phi).$$
\end{conjecture}

We provide a strong evidence for Conjecture \ref{conj.reg}. More precisely, we prove a similar statement to Conjecture \ref{conj.reg} for the $a^*$-invariant, a closely related variant of the regularity. For a closed point $\wp \in \ix$, we define the {\it $a^*$-invariant} of $\bx_\wp$, denoted by $a^*(\bx_\wp)$, to be the $a^*$-invariant of its homogeneous coordinate ring $\R_{(\wp)}$. Our first main result is stated as follows.

\begin{theorem}[Theorems \ref{thm.ainvariant}] \label{intro.thm2}
Let $X = \proj R \subseteq \PP^n$ be a projective scheme, and let $I \subseteq R$ be a homogeneous ideal generated by forms of degree $d$. Let $a^*(\phi) = \max\{ a^*(\bx_\wp) ~|~ \wp \in \ix \}$. Then for $q \gg 0$, we have
$$a^*(I^q) = qd + a^*(\phi).$$
\end{theorem}

As a consequence of Theorem \ref{intro.thm2}, we obtain in Theorem \ref{thm.regbound} an upper and a lower bounds for the asymptotic linear function $\reg(I^q)$. We prove that for $q \gg 0$,
$$qd + a^*(\phi) \le \reg(I^q) \le qd + a^*(\phi) + \dim R.$$
This, in particular, allows us to settle Conjecture \ref{conj.reg} an important case. A fiber $\bx_\wp$ is said to be {\it arithmetically Cohen-Macaulay} if its homogeneous coordinate ring $\R_\wp$ is Cohen-Macaulay. Our next result shows that Conjecture \ref{conj.reg} holds under the additional condition that each fiber $\bx_\wp$ is arithmetically Cohen-Macaulay. This hypothesis is satisfied, for instance, when the Rees algebra $\R$ is a Cohen-Macaulay ring. 

\begin{theorem}[Theorem \ref{thm.reg}] \label{intro.thm1}
Let $X = \proj R \subseteq \PP^n$ be an irreducible projective scheme of dimension at least 1, and let $I \subseteq R$ be a homogeneous ideal generated by forms of degree $d$. Let $\reg(\phi) = \max \{ \reg(\bx_\wp) ~|~ \wp \in \ix\}$. Assume that each fiber $\bx_\wp$ is an arithmetically Cohen-Macaulay scheme. Then for $q \gg 0$, we have
$$\reg(I^q) = qd + \reg(\phi).$$
\end{theorem}

Finally, we use our method to prove a special case of the following conjecture of Beheshti and Eisenbud \cite{BE}.

\begin{conjecture}[Beheshti-Eisenbud] \label{conj}
Let $R$ be a standard graded $k$-algebra of dimension $(l+1)$, and let $\m$ be the maximal homogeneous ideal of $R$. Suppose that $R$ is a domain with isolated singularity. If $k$ is infinite and $I \subseteq R$ is an ideal generated by $(l+1+c)$ general linear forms, then for $\epsilon = \lfloor l/c \rfloor$,
$$\m^{q+\epsilon} \subseteq I^q \text{ for } q \gg 0.$$
\end{conjecture}

Eisenbud and Harris \cite[Theorem 0.1]{EH} gave a nice translation of Conjecture \ref{conj} into a problem about the asymptotic linearity of $\reg(I^q)$; that is to show that $\reg(I^q) \le q + \epsilon$ for $q \gg 0$. Based on this translation, they proved Conjecture \ref{conj} for $c = 1$ and $l \le 14$ (in \cite[Corollary 0.4]{EH}). We prove Conjecture \ref{conj} for $c = 1$ and any $l$.

\begin{theorem}[Theorem \ref{thm.gen}] \label{intro.thm3}
Let $R$ be a standard graded $k$-algebra of dimension $(l+1)$. Let $I = (L_0, \dots, L_{l+1}) \subseteq R$, where $L_i$s are linear forms whose zeros set is empty in $X = \proj R$. Then for $q \gg 0$, we have
$$\reg(I^q) \le q + l.$$
\end{theorem}

Our method in proving Theorem \ref{intro.thm2}, and subsequently Theorem \ref{intro.thm1}, is based upon investigating different graded structures of the Rees algebra $\R$. More precisely, beside the natural bi-graded structure mentioned above, $\R$ possesses two other $\NN$-graded structures; namely
\begin{align*}
\R &  = \bigoplus_{q \in \ZZ} \R^1_q, \text{ where } \R^1_q = \bigoplus_{p \in \ZZ} \R_{(p,q)}, \text{ and } \\
\R & = \bigoplus_{p \in \ZZ} \R^2_p, \text{ where } \R^2_p = \bigoplus_{q \in \ZZ} \R_{(p,q)}.
\end{align*} 
Under these $\NN$-graded structures, it can be seen that $R = \R^1_0 \hookrightarrow \R$, $\R^1_q$ is a graded $R$-modules for any $q \in \ZZ$, $S = \R^2_0 \hookrightarrow \R$, and $\R^2_p$ is a graded $S$-modules for any $p \in \ZZ$. Let $\widetilde{\R^1_q}$ be the coherent sheaf associated to $\R^1_q$ on $X$, and let $\widetilde{\R^2_p}$ be the coherent sheaf associated to $\R^2_p$ on $\ix$. Observe further that $\R^1_q = \bigoplus_{p \in \ZZ} \big[I^q\big]_{p+qd} = I^q(qd)$, the module $I^q$ shifted by $qd$. As a consequence, for any $p, q \in \ZZ$ we have
$$\widetilde{\R^1_q}(p) = \widetilde{I^q}(p+qd).$$
Thus, to study the regularity of $I^q$, we examine sheaf cohomology groups of $\widetilde{\R^1_q}(p)$. Our results are obtained by investigating how these sheaf cohomology groups behave by pulling back via the blowup map $\pi$ and pushing forward through the projection map $\phi$.

To prove Theorem \ref{intro.thm3}, we make use of Theorem \ref{thm.regbound} to bring the problem to showing that $a^*(\phi) \le -1$. This is then accomplished by proving that for any $i \ge 0$ and any $\wp \in \ix$, $a^i(\R_{(\wp)}) \le -1$.

Our paper is outlined as follows. In the next section, we consider $\bx$ as a biprojective scheme and prove a similar statement to Conjecture \ref{conj.reg} for the $a^*$-invariant. In the last section, we prove special cases of Conjectures \ref{conj.reg} and \ref{conj}. 

\noindent{\bf Acknowledgement.} The author thanks C. Polini and B. Ulrich for stimulating discussions on the topic, and thanks D. Eisenbud for explaining their results in \cite{EH}. The author is partially supported by Board of Regents Grant LEQSF(2007-10)-RD-A-30 and Tulane's Research Enhancement Fund.

%%%%%%%%%%%%%%%%%%%%%%%%%%%%%%%%%%%%%%%%%%%%%%%%%%%%%%%

\section{Bi-projective schemes and $a^*$-invariants} \label{section.ainv}

The goal of this section is to give a similar statement to Conjecture \ref{conj.reg} for the $a^*$-invariant of powers of an ideal.
We first recall the definition of regularity and $a^*$-invariant.

\begin{defn} For any $\NN$-graded algebra $T$, let $T_+$ denote its irrelevant ideal. For $i \ge 0$, let
$a^i(T) = \max \{ l ~|~ [H^i_{T_+}(T)]_l \not= 0\}$ (if $H^i_{T_+}(T) = 0$ then take $a^i(T) = -\infty$). The {\it $a^*$-invariant} and the {\it regularity} of $T$ are defined to be
$$a^*(T) = \max_{i \ge 0} \{ a^i(T) \} \text{ and } \reg(T) = \max_{i \ge 0} \{ a^i(T) + i\}.$$
\end{defn}
\noindent Note that $H^i_{T_+}(T) = 0$ for $i > \dim T$, so $a^*(T)$ and $\reg(T)$ are well-defined and finite invariants.

Let $S$ denote the homogeneous coordinate ring of $\ix \subseteq \PP^m$. For each closed point $\p \in \ix$, i.e., $\p$ is a homogeneous prime ideal in $S$, let $\R_\p$ be the localization of $\R$ at $\p$; that is, $\R_\p = \R \otimes_S S_\p$. The {\it homogeneous localization} of $\R$ at $\p$, denoted by $\R_{(\p)}$, is the collection of all element of degree 0 (in $t$) of $\R_\p$. 
Observe that homogeneous localization at $\p$ annihilates the grading with respect to powers of $t$. Thus, inheriting from the bi-graded structure of $\R$, the homogeneous localization $\R_{(\wp)}$ is a $\NN$-graded ring. The regularity and $a^*$-invariant of $\R_{(\wp)}$ are therefore defined as usual. 

Associated to $\phi: \bx \rightarrow \ix$, let
\begin{align*}
a^i(\phi) & = \max \{a^i(\R_{(\wp)}) ~|~ \wp \in \ix\} \text{ for } i \ge 0, \\
a^*(\phi) & = \max \{ a^*(\R_{(\wp)}) ~|~ \wp \in \ix\}, \text{ and } \\
\reg(\phi) & = \max \{ \reg(\R_{(\wp)}) ~|~ \wp \in \ix\}.
\end{align*}

\begin{remark} \label{rmk.finite}
By definition, $a^*(\phi) = \max_{i \ge 0} \{a^i(\phi)\}$ and $\reg(\phi) = \max_{i \ge 0} \{a^i(\phi) + i\}$. Note that $H^i_{\R_{(\p)+}}(\R_{(\p)}) = \big[H^i_{\R_+}(\R)\big]_{(\p)}$, where on the right hand side we view $\R$ under its $\NN$-graded structure $\R = \bigoplus_{p \in \ZZ} \R^2_p$, which induces the embedding $S \hookrightarrow \R$. Thus, $a^i(\phi)$ is a well-defined and finite invariant for any $i \ge 0$. As a consequence, $a^*(\phi)$ and $\reg(\phi)$ are well-defined and finite invariants. These invariants are defined in a similar fashion to the {\it projective $a^*$-invariant} that was introduced in \cite{HT}. We shall also let $r_\phi$ denote the smallest integer $r$ such that 
$$a^*(\phi) = a^r(\phi).$$
\end{remark}

Recall further that the Rees algebra $\R = R[It]$ of $I$ is naturally bi-graded with $\R_{(p,q)} = (I^q)_{p+qd} t^q$, and we identify $\bx$ with $\proj \R \subseteq \PP^n \times \PP^m$. It can also be seen that $\bx$ and $\ix$ can be realized as the (closure of the) graph and the (closure of the) image of the rational map $\varphi: X \dashrightarrow \PP^m$ given by $P \mapsto [F_0(P): \dots: F_m(P)]$ (cf. \cite{CH, HT}). Under this identification, $\pi$ and $\phi$ are restrictions on $\bx$ of natural projections $\PP^n \times \PP^m \rightarrow \PP^n$ and $\PP^n \times \PP^m \rightarrow \PP^m$. We have the following diagram:
$$\begin{array}{rcccl} & & \bx & \subseteq & \PP^n \times \PP^m \\
\pi & \swarrow & & \searrow & \phi \\
X & & \stackrel{\varphi}{\dasharrow} & & \ix \end{array}$$

Let $\I$ be the ideal sheaf of $I$, and let $\L = \I \O_\bx = \O_\bx(0,1).$

\begin{lemma} \label{lem.grading} 
With notations as above.
\begin{enumerate}
\item The homogeneous coordinate ring of $\ix$ is $S \simeq k[F_0t, \dots, F_mt]$.
\item $\phi^* \O_\ix(q) = \L^q \otimes \pi^* \O_X(qd) \ \forall \ q \in \ZZ.$
\item $\O_\bx(p,q) = \pi^* \O_X(p) \otimes \phi^* \O_\ix(q) \simeq \L^q \otimes \pi^* \O_X(p+qd) \ \forall \ p,q \in \ZZ.$
\end{enumerate}
\end{lemma}

\begin{proof} (1) follows from the construction of $\varphi$. (2) and (3) follow from the graded structures of $R, \R$ and $S$.
\end{proof}

The next few lemmas exhibit how the $a^*$-invariant of fibers of $\phi$ governs sheaf cohomology groups via a push forward along $\phi$.

\begin{lemma} \label{lem.vanishing}
Let $p > a^*(\phi)$. Then
\begin{enumerate}
\item $\phi_* \O_\bx(p,q) = \widetilde{\R^2_p}(q)$ and $R^j \phi_* \O_\bx(p,q) = 0$ for any $j > 0$ and any $q \in \ZZ$,
\item $H^i(\bx, \O_\bx(p,q)) = 0$ for $i > 0$ and $q \gg 0$.
\end{enumerate}
\end{lemma}

\begin{proof} By Lemma \ref{lem.grading} and the projection formula we have 
$$\phi_* \O_\bx(p,q) = \phi_*\O_\bx(p,0) \otimes \O_{\ix}(q) \text{ and } R^j \phi_* \O_\bx(p,q) = R^j \phi_* \O_\bx(p,0) \otimes \O_\ix(q).$$
Thus, to show (1) it suffices to prove the assertion for $q =0$.

Let $\p$ be any closed point of $\ix$, and consider the restriction $\phi_\wp: \bx_\p = \proj \R_{(\p)} \rightarrow \spec \O_{\ix,\p}$ of $\phi$ over an open affine neighborhood $\spec \O_{\ix, \p}$ of $\p$. We have
\begin{align}
R^j \phi_* \O_\bx(p,0) \Big|_{\spec \O_{\ix,\p}} = R^j \phi_* \Big(\widetilde{\R_{(\p)}}(p)\Big) = H^j(\bx_\p, \widetilde{\R_{(\p)}}(p))\!\!\!^{\widetilde{\quad \quad}} \ \forall \ j \ge 0. \label{eq.affine}
\end{align}

For any $j \ge 0$ and any $\p \in \ix$, we have $p > a^*(\phi) \ge a^j(\R_{(\p)})$; and thus, $\big[H^j_{\R_{(\p)+}}(\R_{(\p)})\big]_p = 0$.
Moreover, the Serre-Grothendieck correspondence give us an exact sequence
$$0 \rightarrow \big[H^0_{\R_{(\p)+}}(\R_{(\p)})\big]_p \rightarrow \big[\R_{(\p)}\big]_p = \big(\R^2_p\big)_{(\p)} \rightarrow H^0(\bx_\p, \widetilde{\R_{(\p)}}(p)) \rightarrow \big[H^1_{\R_{(\p)+}}(\R_{(\p)})\big]_p \rightarrow 0$$
and isomorphisms
$$H^i(\bx_\p, \widetilde{\R_{(\p)}}(p)) \simeq \big[H^{i+1}_{\R_{(\p)+}}(\R_{(\p)})\big]_p \text{ for } i > 0.$$
Therefore, for any $j \ge 0$ and any $\p \in \ix$, 
$$R^j \phi_* \O_\bx(p,0) \Big|_{\spec \O_{\ix,\p}} = H^j(\bx_\p, \widetilde{\R_{(\p)}}(p))\!\!\!^{\widetilde{\quad \quad}} = \left\{ \begin{array}{ll} \widetilde{(\R^2_p)_{(\p)}} & \text{ for } j = 0 \\
0 & \text{ for } j > 0. \end{array} \right.$$
This is true for any $\p \in \ix$, and so (1) is proved.

Now, it follows from (1) that the Leray spectral sequence $H^i(\ix, R^j \phi_* \O_\bx(p,q)) \Rightarrow H^{i+j}(\bx, \O_\bx(p,q))$ degenerates. Thus, for any $j \ge 0$,
$$H^j(\bx, \O_\bx(p,q)) = H^j(\ix, \widetilde{\R^2_p}(q)).$$
Moreover, since $\O_\ix(1)$ is a very ample divisor, we have $H^j(\ix, \widetilde{\R^2_p}(q)) = 0$ for all $q \gg 0$, and (2) is proved.
\end{proof}

\begin{lemma} \label{lem.nonvanishing}
Let $r_\phi$ be defined as above.
\begin{enumerate}
\item If $r_\phi \le 1$ then $H^0(\bx, \O_\bx(a^*(\phi),q)) \not= \R_{(a^*(\phi),q)}$ for $q \gg 0$.
\item If $r_\phi \ge 2$ then $H^{r_\phi-1}(\bx, \O_\bx(a^*(\phi),q)) \not= 0$ for $q \gg 0$.
\end{enumerate}
\end{lemma}

\begin{proof} For simplicity, let $a = a^*(\phi)$. By the definition of $r_\phi$, we have
\begin{align}
\left\{ \begin{array}{ll} \big[H^i_{\R_{(\p)+}}(\R_{(\p)})\big]_a = 0 & \text{ for } i < r_\phi \text{ and any } \p \in \ix \\
\big[H^{r_\phi}_{\R_{(\q)+}}(\R_{(\q)})\big]_a \not= 0 & \text{ for some } \q \in \ix. \end{array} \right. \label{eq.nonvanishing}
\end{align}

(1) If $r_\phi \le 1$ then it follows from (\ref{eq.nonvanishing}) and the Serre-Grothendieck correspondence that $H^0(\bx_\q, \widetilde{\R_{(\q)}}(a)) \not= \big[\R_{(\q)}\big]_a = \big(\R^2_a\big)_{(\q)}$. This and (\ref{eq.affine}) imply that $\phi_* \O_\bx(a,0) \not= \widetilde{\R^2_a}$, and so
$$\phi_* \O_\bx(a,q) \not= \widetilde{\R^2_a}(q) \text{ for any } q \in \ZZ.$$
Since both $\phi_* \O_\bx(a,q) = \phi_* \O_\bx(a,0) \otimes \O_\ix(q)$ (by Lemma \ref{lem.grading} and the projection formula) and $\widetilde{\R^2_a}(q)$ are generated by global sections for $q \gg 0$, we must have 
$$H^0(\ix, \phi_* \O_\bx(a,q)) \not= H^0(\ix, \widetilde{\R^2_a}(q)) = \R_{(a,q)} \ \forall \ q \gg 0.$$
Moreover, $H^0(\bx, \O_\bx(a,q)) = H^0(\ix, \phi_* \O_\bx(a,q))$. Thus,
$$H^0(\bx, \O_\bx(a,q)) \not= \R_{(a,q)} \text{ for } q \gg 0.$$

(2) If $r_\phi \ge 2$, then it follows from (\ref{eq.nonvanishing}) and (\ref{eq.affine}) that
\begin{align}
\left\{ \begin{array}{l} R^j \phi_* \O_\bx(a,q) = 0 \text{ for } 0 < j < r_\phi-1, \\
R^{r_\phi-1} \phi_* \O_\bx(a,q) \not= 0. \end{array} \right. \label{eq.higher}
\end{align}
By Lemma \ref{lem.grading} and the projection formula, $\phi_* \O_\bx(a,q) = \phi_* \O_\bx(a,0) \otimes \O_\ix(q)$. Thus, for $q \gg 0$ we have $H^{r_\phi-1}(\ix, \phi_* \O_\bx(a,q)) = 0.$ From this, together with (\ref{eq.higher}) and the Leray spectral sequence
$H^i(\ix, R^j \phi_* \O_\bx(a,q)) \Rightarrow H^{i+j}(\bx, \O_\bx(a,q))$, we can deduce that
$$H^{r_\phi-1}(\bx, \O_\bx(a,q)) = H^0(\ix, R^{r_\phi-1} \phi_* \O_\bx(a,q)) \text{ for } q \gg 0.$$
It then follows, since $R^{r_\phi-1} \phi_* \O_\bx(a,q) = R^{r_\phi-1} \phi_* \O_\bx(a,0) \otimes \O_\ix(q)$ is globally generated for $q \gg 0$, that
$$H^{r_\phi-1}(\bx, \O_\bx(a,q)) \not= 0 \text{ for } q \gg 0.$$
\end{proof}

Our first main result is a similar statement to Conjecture \ref{conj.reg} for the $a^*$-invariant.

\begin{theorem} \label{thm.ainvariant}
Let $X = \proj R \subseteq \PP^n$ be a projective scheme, and let $I \subseteq R$ be a homogeneous ideal generated by forms of degree $d$. Let $a^*(\phi) = \max\{ a^*(\bx_\wp) ~|~ \wp \in \ix \}$. Then for $q \gg 0$, we have
$$a^*(I^q) = qd + a^*(\phi).$$
\end{theorem}

\begin{proof} By a similar argument as in Lemma \ref{lem.vanishing}, considering $\pi_*$ instead of $\phi_*$, we can show that for $q \gg 0$,
\begin{align}
\pi_* \O_\bx(p,q) = \widetilde{\R^1_q}(p) = \widetilde{I^q}(p+qd) \text{ and } R^j \pi_* \O_\bx(p,q) = 0 \ \forall \ j > 0. \label{eq.newLeray}
\end{align}
This implies that for $q \gg 0$, the Leray spectral sequence $H^i(X, R^j \pi_* \O_\bx(p,q)) \Rightarrow H^{i+j}(\bx, \O_\bx(p,q))$ degenerates and we have
$$H^j(\bx, \O_\bx(p,q)) = H^j(X, \widetilde{I^q}(p+qd)) \ \forall \ j \ge 0.$$
Therefore, for $j > 0$, $q \gg 0$ and $p > a^*(\phi)$, it follows from Lemma \ref{lem.vanishing} that $H^j(X, \widetilde{I^q}(p+qd)) = 0$. That is, 
\begin{align}
\big[H^{j+1}_{R_+}(I^q)\big]_{p+qd} = 0. \label{eq.degree}
\end{align}
Furthermore, for $j = 0$ and $q \gg 0$, we have
$H^0(\ix, \widetilde{\R^2_p}(q)) = H^0(\bx, \O_\bx(p,q)) = H^0(X, \widetilde{I^q}(p+qd))$, where the first equality follows from Lemma \ref{lem.vanishing}. 
On the other hand, for $q \gg 0$, $H^0(\ix, \widetilde{\R^2_p}(q)) = (\R^2_p)_q = \R_{(p,q)} = [I^q]_{p+qd}$. Thus, for $q \gg 0$, $H^0(X, \widetilde{I^q}(p+qd)) = [I^q]_{p+qd}$.
This and (\ref{eq.degree}) imply that for $q \gg 0$, 
$$a^*(I^q) \le qd + a^*(\phi).$$

To prove the other inequality, let $r_\phi$ be defined as in Remark \ref{rmk.finite}. We consider two cases: $r_\phi \le 1$ and $r_\phi \ge 2$. If $r_\phi \le 1$ then by Lemma \ref{lem.nonvanishing}, $H^0(\bx, \O_\bx(a^*(\phi), q)) \not= \R_{(a^*(\phi),q)}$ for all $q \gg 0$. This implies that $H^0(X, \pi_* \O_\bx(a^*(\phi),q)) \not= \R_{(a^*(\phi),q)}$ for $q \gg 0$. That is, 
$$H^0(X, \widetilde{I^q}(a^*(\phi)+qd)) \not= \big[I^q\big]_{a^*(\phi)+qd} \ \forall \ q \gg 0.$$ 
By the Serre-Grothendieck correspondence, for $q \gg 0$, we have either $$\big[H^0_{R_+}(I^q)\big]_{(a^*(\phi)+qd,q)} \not= 0 \text{ or } \big[H^1_{R_+}(I^q)\big]_{(a^*(\phi)+qd,q)} \not= 0.$$
It then follows that
$a^*(I^q) \ge qd+a^*(\phi) \text{ for } q \gg 0$.

If $r_\phi \ge 2$, then by Lemma \ref{lem.nonvanishing}, $H^{r_\phi-1}(\bx, \O_\bx(a^*(\phi),q)) \not= 0$ for $q \gg 0$. Moreover, for $q \gg 0$, it follows from (\ref{eq.newLeray}) that the Leray spectral sequence 
$$H^i(X, R^j \pi_* \O_\bx(p,q)) \Rightarrow H^{i+j}(\bx, \O_\bx(p,q))$$ 
degenerates. Thus, for $q \gg 0$, we have $H^{r_\phi-1}(X, \widetilde{I^q}(a^*(\phi)+qd)) \not= 0$. By the Serre-Grothendieck correspondence, we have $\big[H^{r_\phi}_{R_+}(I^q)\big]_{a^*(\phi)+qd} \not=0$ for $q \gg 0$. This implies that $a^*(I^q) \ge qd+a^*(\phi)$ for $q \gg 0$.
\end{proof}

\begin{example} \label{example.minor}
Let $R = \bigoplus_{n \ge 0}R_n$ be a Cohen-Macaulay standard graded domain, and let $A = (a_{ij})_{1\le i \le r, 1\le j \le s}$ be an $r \times s$ matrix ($r \le s$) of entries in $R_1$. Let $I_t(A)$ denote the ideal generated by $t \times t$ minors of $A$, and let $I = I_r(A)$. Assume that for any $1\le t \le r$, $\height I_t(A) \ge (r-t+1)(s-r)+1$. Let $\imath(\omega_R)$ be the least generating degree of $\omega_R$, the canonical module of $R$. Then for $q \gg 0$,
$$a^*(I^q) = qr - \imath(\omega_R).$$
Indeed, let $S = k[It]$ denote the homogeneous coordinate ring of $\ix$, let $\wp$ be any point in $\ix$, and let $T = \R_{(\wp)}$. By \cite[Theorem 3.5]{EHu}, the Rees algebra $\R$ is Cohen-Macaulay. Thus, $\R_{(\wp)}$ is Cohen-Macaulay. This implies that
$$a^*(T) = a^{\dim T}(T) = - \min \{ s ~|~ [\omega_T]_s \not= 0\}.$$ 
Furthermore, by \cite[Example 3.8]{HSV}, 
$$\omega_\R = \omega_R(1,t)^{g-2} = \omega_R \oplus \omega_Rt \oplus \dots \oplus \omega_Rt^{g-2} \oplus \omega_RIt^{g-1} \oplus \dots,$$
where $g = \height I$. Hence, by localizing at $\wp$, we obtain 
$$\omega_T = \big(\omega_\R\big)_{(\wp)} = \big(\omega_R(1,t)^{g-2}\big)_{(\wp)}.$$ 
Observe that the homogeneous localization at $\wp$ annihilates the grading inherited from powers of $t$, so it follows that the degrees of $\omega_T$ arise from the degrees of $\omega_R$. That is, $\imath(\omega_T) = \imath(\omega_R)$, and the conclusion follows from Theorem \ref{thm.ainvariant}.
\end{example}

%%%%%%%%%%%%%%%%%%%%%%%%%%%%%%%%%%%%%%%%%%%%%%%%%%%%%%%%%%%%%%

\section{Regularity of powers of ideals and generic projections} \label{section.reg}

In this section, we investigate the asymptotic linearity of regularity and prove an important special case of Conjecture \ref{conj.reg}. As an application, we also show that Conjecture \ref{conj} holds when $c = 1$.

We start by giving an upper and a lower bound for the free constant of $\reg(I^q)$ in terms of $a^*(\phi)$.

\begin{theorem} \label{thm.regbound}
Let $X = \proj R \subseteq \PP^n$ be a projective scheme, and let $I \subseteq R$ be a homogeneous ideal generated by forms of degree $d$. Let $a^*(\phi) = \max\{ a^*(\bx_\wp) ~|~ \wp \in \ix \}$. Then there exists an integer $0 \le r \le \dim R$ such that for $q \gg 0$, we have $\reg(I^q) = qd + a^*(\phi) + r.$ In particular, for $q \gg 0$,
$$qd + a^*(\phi) \le \reg(I^q) \le qd + a^*(\phi) + \dim R.$$
\end{theorem}

\begin{proof} Suppose $\reg(I^q) = aq + b$ for $q \gg 0$. It can be easily seen from the definition of the regularity and $a^*$-invariant of graded $R$-modules that $a^*(I^q) \le \reg(I^q) \le a^*(I^q) + \dim R$ for any $q$. This and Theorem \ref{thm.ainvariant} imply that $a = d$; that is, $\reg(I^q) = qd + b$ for $q \gg 0$. Let $r = b - a^*(\phi)$. Then $\reg(I^q) = qd + a^*(\phi) +r$, and since $a^*(I^q) \le \reg(I^q) \le a^*(I^q) + \dim R$, we have $0 \le r \le \dim R$.
\end{proof}

Our next result shows that Conjecture \ref{conj.reg} holds under an extra condition that each fiber $\bx_\wp$ is an arithmetically Cohen-Macaulay scheme.

\begin{theorem} \label{thm.reg}
Let $X = \proj R \subseteq \PP^n$ be an irreducible projective scheme of dimension at least 1, and let $I \subseteq R$ be a homogeneous ideal generated by forms of degree $d$. Let $\reg(\phi) = \max \{ \reg(\bx_\wp) ~|~ \wp \in \ix\}$. Assume that each fiber $\bx_\wp$ is an arithmetically Cohen-Macaulay scheme. Then for $q \gg 0$, we have
$$\reg(I^q) = qd + \reg(\phi).$$
\end{theorem}

\begin{proof} Let $l = \dim X \ge 1$. Since $X$ is irreducible, $\bx$ is also irreducible. Moreover, for any point $\wp \in \ix$, $\spec \O_{\ix, \wp}$ is an open neighborhood of $\wp$, and so $\bx_\wp = \phi^{-1}(\spec \O_{\ix, \wp})$ is an open subset in $\bx$. Thus, $\dim \bx_\wp = \dim \bx = \dim X$.

By the hypothesis, for each $\wp \in \ix$, $\R_{(\wp)}$ is a Cohen-Macaulay ring of dimension $\dim \bx_\wp + 1 = l+1$. This implies that $a^*(\R_{(\wp)}) = a^{l+1}(\R_{(\wp)})$ and $\reg(\R_{(\wp)}) = a^{l+1}(\R_{(\wp)}) + (l+1)$. Therefore, 
\begin{align}
a^*(\phi) & = a^{l+1}(\phi), \label{eq.ainv} \\
\reg(\phi) & = a^*(\phi) + l+1. \label{eq.reg}
\end{align}
It follows from (\ref{eq.ainv}) that $r_\phi = l+1 \ge 2$. By the same arguments as the last part of the proof of Theorem \ref{thm.ainvariant}, we have that for $q \gg 0$, $\reg(I^q) \ge qd + a^*(\phi) + r_\phi = qd + a^*(\phi) + \dim R$. This, together with Theorem \ref{thm.regbound}, implies that for $q \gg 0$, $\reg(I^q) = qd + a^*(\phi) + \dim R.$ The conclusion now follows from (\ref{eq.reg}).
\end{proof}

\begin{corollary} \label{cor.reg}
Let $X = \proj R \subseteq \PP^n$ be an irreducible projective scheme of dimension at least 1, and let $I \subseteq R$ be a homogeneous ideal generated by forms of degree $d$. Assume that $\R$ is a Cohen-Macaulay ring. Then for $q \gg 0$,
$$\reg(I^q) = qd + \reg(\phi).$$
\end{corollary}

\begin{proof} Since $\R$ is Cohen-Macaulay, so is $\R_{(\wp)}$ for any $\wp \in \ix$. Thus, each fiber $\bx_\wp$ is arithmetically Cohen-Macaulay. The conclusion follows from Theorem \ref{thm.reg}.
\end{proof}

\begin{example} Let $R$ and $I$ be as in Example \ref{example.minor}. In this case, $I$ is generated in degree $r$. As noted before, the Rees algebra $\R$ is Cohen-Macaulay. Notice further that $X = \proj R$ is an irreducible projective scheme. Thus, by Corollary \ref{cor.reg}, we have
$$\reg(I^q) = qr + \reg(\phi) \ \forall \ q \gg 0.$$
\end{example}

\begin{example} Let $R = k[x_{ij}]_{1 \le i \le r, 1 \le j \le s}$ and let $I$ be the ideal generated by $t \times t$ minors of $M = (x_{ij})_{1 \le i \le r, 1 \le j \le s}$ for some $1 \le t \le \min \{r,s\}$. By \cite[Theorem 3.5]{EHu} and \cite[Corollary 3.3]{Bruns}, the Rees algebra $\R$ of $I$ is Cohen-Macaulay. Also, $X = \proj R$ is an irreducible projective scheme. It follows from Corollary \ref{cor.reg} that
$$\reg(I^q) = qt + \reg(\phi) \ \forall \ q \gg 0.$$
\end{example}

\begin{example} Let $R$ be a Cohen-Macaulay graded domain of dimension at least 2. Let $I$ be either a complete intersection, or an almost complete intersection that is also generically a complete intersection. Assume that $I$ is generated in degree $d$. Then the Rees algebra $\R$ of $I$ is Cohen-Macaulay (cf. \cite{B,V}). By Corollary \ref{cor.reg}, we have 
$$\reg(I^q) = qd + \reg(\phi) \ \forall \ q \gg 0.$$
\end{example}

We now prove a special case of Conjecture \ref{conj} (when $c = 1$).

\begin{theorem} \label{thm.gen}
Let $R$ be a standard graded $k$-algebra of dimension $(l+1)$. Let $I = (L_0, \dots, L_{l+1}) \subseteq R$, where $L_i$s are linear forms whose zeros set is empty in $X = \proj R$. Then for $q \gg 0$, 
$$\reg(I^q) \le q + l.$$
\end{theorem}

\begin{proof} 
By Theorem \ref{thm.regbound}, to prove our assertion it suffices to show that $a^*(\phi) + \dim R \le l$; that is, $a^*(\phi) \le -1$. 

Since the zeros set of $L_i$s is empty, we have $\bx \simeq X$. Also, it follows from \cite{BE} that the fiber over any point in $\ix$ has dimension 0. Thus, by \cite[Corollary III.11.2]{Hartshorne}, $R^j \phi_* (\O_\bx(p,0)) = 0$ for any $j > 0$. In particular, for any $\wp \in \ix$, $H^j(\bx_\wp, \widetilde{\R_{(\wp)}}(p))\!\!\!^{\widetilde{\quad \quad}} = 0$ for all $j > 0$. Since $\spec \O_{\ix, \wp}$ is an affine scheme, we have $H^j(\bx_\wp, \widetilde{\R_{(\wp)}}(p)) = 0$ for all $j > 0$. This implies that for any $\wp \in \ix$, $\big[H^{j+1}_{\R_{(\wp)}+}(\R_{(\wp)})\big]_p = 0$ for all $j> 0$. Thus, $a^j(\phi) = -\infty$ for all $j \ge 2$.

It remains to show that $a^0(\phi) \le -1$ and $a^1(\phi) \le -1$. By the Serre-Grothendieck correspondence between local and sheaf cohomology, this is equivalent to showing that for any $\wp \in \ix$, the natural map $(\R^2_p)_{(\wp)} \rightarrow H^0(\bx_\wp, \widetilde{\R_{(\wp)}}(p))$ is an isomorphism for any $p \ge 0$. For $p = 0$, we have $H^0(\bx_\wp, \widetilde{\R_{(\wp)}}) = H^0(\bx_\wp, \O_{\bx_\wp}) = H^0(\spec \O_{\ix, \wp}, \phi_* (\O_{\bx_\wp})) = H^0(\spec \O_{\ix, \wp}, \O_{\ix, \wp}) = S_{(\wp)}$, where $S$ denotes the homogeneous coordinate ring of $\ix$ in $\PP^{l+1}$. By Lemma \ref{lem.grading}, we now have $H^0(\bx_\wp, \widetilde{\R_{(\wp)}}) = k[L_0t, \dots, L_{l+1}t]_{(\wp)} = (\R^2_0)_{(\wp)}$. For $p \ge 1$, since $\R_{(\wp)}$ is generated in degree $1$ (because $R$ is a standard graded algebra), we also have $H^0(\bx_\wp, \widetilde{\R_{(\wp)}}(p)) = (\R_{(\wp)})_p = (\R^2_p)_{(\wp)}$. The theorem is proved.
\end{proof}

%%%%%%%%%%%%%%%%%%%%%%%%%%%%%%%%%%%%%%%%%%%%%%%%%%%%%%%

\end{document}